\documentclass[12pt,reqno]{amsart}
\advance \topmargin by -\headheight
\advance \topmargin by -\headsep
\evensidemargin \oddsidemargin
\newtheorem{theorem}{Theorem}[section]
\newtheorem{lemma}[theorem]{Lemma}

\usepackage{amsmath, amsthm, amssymb, amsfonts}
 \usepackage{hyperref}
\usepackage{enumerate}
\usepackage{url}
\usepackage{dsfont}
\usepackage{mathrsfs}

\usepackage{bm}

\usepackage{xcolor}

\usepackage{fancyvrb}

\newtheorem{Proposition}[theorem]{Proposition}

\newtheorem{Corollary}[theorem]{Corollary}

\theoremstyle{definition}

\newtheorem{example}[theorem]{Example}

\theoremstyle{remark}

\numberwithin{equation}{section}

\renewcommand{\geq}{\geqslant}
\renewcommand{\leq}{\leqslant}

\def \Z {\mathbb{Z}}

\newcommand\1{\mathds{1}}

\title{On an asymmetric additive energy inequality}

\author{Akshat Mudgal}
\address{Mathematics Institute, Zeeman Building, University of Warwick, Coventry CV4 7AL, UK}
\email{Akshat.Mudgal@warwick.ac.uk}


\renewcommand\vec{\bm}

\begin{document}
\maketitle

\begin{abstract}
Let $d \geq 1$ be an integer, $G$ be an abelian group and $\nu, w_1, \dots, w_{2d}: G \to [0, \infty)$ be functions with finite, non-empty supports. Define the generalised additive energy
\[ E_{2d, \nu}(w_1, \dots, w_{2d}) = \sum_{y,y' \in G}\sum_{a_1, \dots, a_{2d} \in G } w_1(a_1) \dots w_{2d}(a_{2d}) \nu(y) \nu(y')\1_{\sum_{i=1}^d (a_i - a_{i+d}) = y-y'} .\]
Moreover, for every $1 \leq i \leq 2d$, let $E_{2d, \nu}(w_i) = E_{2d, \nu}(w_i, \dots, w_i)$. A standard Fourier analytic argument delivers the estimate
\[ E_{2d,\nu}(w_1, \dots, w_{2d}) \leq \prod_{1 \leq i \leq 2d} E_{2d, \nu}(w_i)^{1/2d}.\]
In this note, we present a purely combinatorial proof of the above inequality. In particular, our proof does not use any Fourier or spectral analysis and relies on repeated applications of Cauchy--Schwarz inequality combined with a discrete convexity extension type argument.  We also record a variation of this upper bound in the non-abelian setting via spectral inequalities following work of Hatami on graph norms, as well as a relevant sumset analogue obtained via iterative applications of the Pl\"{u}nnecke--Ruzsa inequality.
\end{abstract}

 \section{Introduction}

Given integer $d \geq 1$, some abelian group $G$ and some functions $\nu, w_1, \dots, w_{2d}: G \to [0, \infty)$ with finite, non-empty supports, we define the generalised additive energy
\begin{equation} \label{chicago}
E_{2d, \nu}(w_1, \dots, w_{2d}) = \sum_{y,y' \in G}\sum_{a_1, \dots, a_{2d} \in G } w_1(a_1) \dots w_{2d}(a_{2d}) \nu(y) \nu(y')\1_{\sum_{i=1}^d (a_i - a_{i+d}) = y-y'} .
\end{equation}
For every $1 \leq i \leq 2d$, let $E_{2d, \nu}(w_i) = E_{2d, \nu}(w_i, \dots, w_i)$. In this note, we are interested in bounding $E_{2d, \nu}(w_1, \dots, w_{2d})$ in terms of $E_{2d, \nu}(w_1), \dots, E_{2d, \nu}(w_{2d})$. Such inequalities have frequently appeared in the study of Waring's problem \cite{Wo1992}, Vinogradov's mean value theorem \cite{Wo2016, Wo2019}, the sum-product phenomenon \cite{Bo2005, BC2004, Ch2003, HRNZ2020, Mu2024a, PZ2020} and other related topics \cite{Shk2021}. Most of these works have required such bounds in the case when $G = \mathbb{Z}/p\mathbb{Z}, \mathbb{Z}$ or $\mathbb{Q}^{\times}$ and have relied on Fourier analytic arguments. We record this in considerable generality below.

\begin{theorem} \label{main}
   Let $d \geq 1$ be an integer, let $G$ be an abelian group and let $\nu, w_1, \dots, w_{2d}: G \to [0, \infty)$ be functions with finite, non-empty supports. Then
    \begin{equation} \label{lmd}
    E_{2d, \nu}(w_1, \dots, w_{2d}) \leq \prod_{i=1}^{2d} E_{2d, \nu}(w_i)^{1/2d}. \end{equation}
\end{theorem}


One can prove this very quickly by passing to the dual group via orthogonality, applying H\"{o}lder's inequality there, and subsequently returning to the original group via another application of orthogonality, see \S\ref{fouriersection} for a self-contained proof of this. On the other hand, all the objects involved in \eqref{lmd} are defined in the original group $G$, and so, it is natural to ask whether there is a purely combinatorial proof of Theorem \ref{main}, that is, whether there is a way to prove \eqref{lmd} without passing to a dual setting. The main aim of this note is to present a purely combinatorial proof of Theorem \ref{main} which does not rely on any Fourier analysis or spectral inequalities. Second, we will also record related results in non-abelian settings and for sumsets via different methods. 

Returning to our first goal, we observe that a simple, combinatorial proof can be written for Theorem \ref{main} when $d = 2^k$ for some $k \in \mathbb{N}$, see for example the $d=2$ case being illustrated below.

\begin{example} \label{inmymind}
For any $1 \leq i,j \leq 4$ and any $x\in G$, let 
\[r_{i,j}(x) = \sum_{a,b,y \in G} w_i(a) w_j(b) \nu(y) \1_{a+b + y =x} \ \ \text{and} \ \  r_{i,j}'(x) = \sum_{a,b,y \in G} w_i(a) w_j(b) \nu(y) \1_{a-b + y =x} \]
Then two applications of Cauchy's inequality and double-counting give us 
\begin{align*}
    E_{4, \nu}(w_1, \dots, w_4)  &  = \sum_{x \in G} r_{1,2}(x) r_{3,4}(x) \leq  (\sum_{x \in G} r_{1,2}(x)^2 )^{1/2}(\sum_{x \in G} r_{3,4}(x)^2 )^{1/2} \\
    & = (\sum_{x \in G} r_{1,1}'(x) r_{2,2}'(x) )^{1/2}(\sum_{x \in G} r_{3,3}'(x) r_{4,4}'(x) )^{1/2} \\
    & \leq \prod_{1 \leq i \leq 4} (\sum_{x \in G} r_{i,i}'(x)^2 )^{1/4} = \prod_{1 \leq i \leq 4}E_{4, \nu}(w_i)^{1/4}. 
\end{align*}
\end{example}

The above proof relies very strongly on the special symmetries that one obtains when $d = 2^k$ for some $k \in \mathbb{N}$. It turns out that the appropriate framework for generalising this for arbitrary $d \in \mathbb{N}$ requires proving a discrete convexity extension type inequality. In order to state the latter, we define
  \begin{equation} \label{hdbleh}
        H_d = \{ (a_1, \dots, a_d) \in \Z^d: a_1, ..., a_d \geq 0  \  \text{and} \ a_1 + \dots + a_d = d\}. 
    \end{equation}
Moreover, for every $1 \leq i \leq d$, we write $\vec{e}_i = (e_{i,1}, \dots, e_{i,d})$ to be the vector satisfying $e_{i,i} = 1$ and $e_{i,j} = 0$ for all $1 \leq j \leq d$ with $j \neq i$. In particular, $\{\vec{e}_1, \dots, \vec{e}_d\}$ forms the canonical basis of $\mathbb{Z}^d$.

\begin{Proposition} \label{q1}
    Let $d$ be a positive integer and let $f : H_d \to [0, \infty)$ be a function such that for any $\vec{u}, \vec{v} \in H_d$ satisfying $(\vec{u}+\vec{v})/2 \in H_d$, we have 
    \begin{equation} \label{disc}
    f( (\vec{u}+\vec{v})/2 ) \leq (f(\vec{u}) + f(\vec{v}))/2. 
    \end{equation}
    Then for any $\vec{u} = (u_1, \dots, u_d) \in H_d$, we have
\begin{equation} \label{lvd}
f(\vec{u}) \leq \frac{1}{d} \sum_{i=1}^d u_i f(d \vec{e}_i).  
\end{equation}
\end{Proposition}

This can be interpreted as extending a discrete midpoint convexity of the form \eqref{disc} to global discrete convexity of the form \eqref{lvd}. Such results are classical if we replace the discrete setting of $H_d$ by the continuous setting of $\mathbb{R}^d$. Furthermore, while there has been much work surrounding the topic of discrete convex analysis, see \cite{MMTT2020} and the references therein, we were unable to find a previous reference for Proposition \ref{q1}.

Due to the combinatorial nature of our argument, one might wonder whether it extends to the non-abelian setting where one may not have access to Fourier analysis.  While unfortunately this is not the case, we can extend a specific version of Theorem \ref{main} to arbitrary finite, non-abelian groups $G$ by employing spectral inequalities in the spirit of the Fourier analytic argument of \S\ref{fouriersection}.  In order to begin this discussion, we first have to define what is a suitable analogue of $E_{2d,\nu}(w_1, \dots, w_{2d})$ in the non-abelian setting, since there are multiple ways to generalise this. One such generalisation is obtained by defining, for any group $G$ and for all functions $w_1, \dots, w_{2d} : G \to [0, \infty)$ with finite, non-empty supports, the quantity
\[ S_{2d}(w_1, \dots, w_{2d}) = \sum_{a_1, \dots, a_{2d} \in G} w_1(a_1) \dots w_{2d}(a_{2d}) \1_{a_1 a_2^{-1}a_3 a_4^{-1} \dots a_{2d-1}a_{2d}^{-1} = 1}. \]
Here, we denote the group operation in $G$ by the dot product and the identity element in $G$ by $1$. We further define $S_{2d}(w_i) = S_{2d}(w_i, \dots, w_i)$ for all $1 \leq i \leq 2d$. Note that when $G$ is abelian, this is precisely $E_{2d, \nu}(w_1, \dots, w_{2d})$ with  $\nu = \1_{\{0\}}$. As previously mentioned, the trick recorded in Example \ref{inmymind} allows us to prove an analogue of Theorem \ref{main} for $S_{2d}(w_1, \dots, w_{2d})$ when $d = 2^k$ for some $k \in \mathbb{N}$, see for example \cite[Lemma 2.6]{Mu2024} or \cite[Lemma 10]{Shk2021}, but this does not seem to extend to arbitrary $d \in \mathbb{N}$. 

It turns out that for all $d \in \mathbb{N}$ and any finite group $G$, we can interpret $S_{2d}(w_1, \dots, w_{2d})$ as a normalised variant of a graph norm for cycles of even length, see \eqref{astro}. The latter type of objects have been well-studied, in part due to their connections to an important question in graph theory known as Sidorenko's conjecture, see work of Hatami \cite{Ha2010}. Following the circle of ideas in \cite[\S2]{Ha2010}, we see that the aforementioned graph norms for cycles of even length can be written as a trace of a product of $2d$ matrices, which, in spirit, feels very similar to the orthogonality relation described in \eqref{wcw}. At this point, one can apply a combination of von Neumann's trace inequality and H\"{o}lder's inequality for Schatten norms to bound the aforementioned trace by a product of Schatten norms of the individual matrices; this step being very similar to \eqref{wcw2}. Finally we rewrite the individual Schatten norms via the trace formula as the desired quantity $S_{2d}(w_i, \dots, w_i)$. This process delivers the following result.


\begin{theorem} \label{nonab}
    For any $d \in \mathbb{N}$, any finite group $G$ and any functions $w_1, \dots, w_{2d} : G \to [0, \infty)$ with non-empty support, we have
    \[ S_{2d}(w_1, \dots, w_{2d}) \leq \prod_{i=1}^{2d} S_{2d}(w_i)^{1/2d}. \]
\end{theorem}

We remark that analysing additive energies via spectral inequalities is a well--trodden path in additive combinatorics, see work of Shkredov on the higher energy method \cite{Shk2013a, Shk2013}. Another relationship between additive energy and traces of certain matrices appears in the breakthrough work of Guth--Maynard, see \cite[\S2]{GM2024}.

We will now record a related estimate on the sumset side. In order to motivate this, let $c \in \mathbb{R}$ and $d, N\in \mathbb{Z}$ satisfy $c,d,N \geq 1$. Moreover, let $A_1, \dots, A_{d}$ be finite, non-empty sets in some abelian group such that $|A_i| =N$ and
\[E_{2d}(A_i) = \sum_{a_1, \dots, a_{2d} \in A_i} \1_{\sum_{i=1}^d (a_i - a_{i+d})=0} \leq N^{2d - c} \]
 for all $1 \leq i \leq d$. We also define the sumsets
\[ A_1 + \dots + A_d = \{ a_1 + \dots + a_d: a_i \in A_i \ \  ( 1\leq i \leq d)\}  \ \ \text{and} \ \ dA_i = \{a_1 + \dots + a_d : a_1, \dots, a_d \in A_i \}.\]
Then Theorem \ref{main} implies that
\[ E_{2d}(A_1, \dots, A_d, A_1, \dots, A_d) = \sum_{a_1, a_1' \in A_1} \dots \sum_{a_d, a_d' \in A_d} \1_{\sum_{i=1}^d (a_i - a_i') = 0} \leq \prod_{1 \leq i \leq d}  E_{2d}(A_i)^{1/d} \leq N^{2d - c}.\]
One may now apply Cauchy's inequality to deduce that
\begin{equation} \label{ohtosee}
    |A_1 + \dots + A_d| \geq \frac{|A_1|^2 \dots |A_d|^2}{E_{2d}(A_1, \dots, A_d, A_1, \dots, A_d)} \geq N^c.
\end{equation}

It is natural to ask if such an estimate can be derived from the weaker hypothesis that $|dA_i| \geq N^c$ for all $1 \leq i \leq d$.  This is indeed weaker since as in the preceding discussion, one sees that Cauchy's inequality combined with the hypothesis $E_{2d}(A_i) \leq N^{2d -c}$ immediately delivers the bound $|dA_i| \geq N^c$. Such an extension fails spectacularly. Indeed, from work of Ruzsa \cite{Ru1992}, we know that there exists some constant $c_0>0$ along with arbitrarily large sets $A \subseteq \mathbb{Z}$ such that 
\[ |A-A| \leq |A|^{2- c_0} \ \ \text{and} \ \ |A+A| \gg |A|^2, \]
where we write $A-A = \{ a- b : a,b \in A\}$. Setting $A_1 = A$ and $A_2 = -A$, we see that
\[ |A_1 + A_2| = |A-A| \leq |A|^{2 - c_0} \ \ \text{while} \ \ |A_1| = |A_2| = |A| \ \ \text{and} \ \ |2A_1|, |2A_2| \gg |A|^2. \]
Despite this counterexample, one can prove the following weaker inequality.

\begin{theorem} \label{prapp}
    Let $d\geq 2$, let $G$ be some abelian group, let $A_1, \dots, A_d$ be finite, non-empty subsets of $G$. Then
    \[ |A_1 + \dots + A_d| \geq (|d A_1| \dots |d A_d|)^{1/2d}. \]
\end{theorem}

This can be deduced in a straightforward manner by iteratively applying the Pl\"{u}nnecke--Ruzsa inequality, see \cite{Pe2012} for a nice proof of the latter. While Theorem \ref{prapp} is trivial when $d=2$, it certainly delivers  non-trivial bounds when $d \geq 3$ and $|dA_i| \geq |A_i|^{c} = N^{c}$ for some $c >2$. For instance, for every $k \in \mathbb{N}$ and for all finite, non-empty sets $A,A_1, \dots, A_k \subseteq \mathbb{R}$, define the product sets
\[ A^{(k)} = \{a_1 \dots a_k : a_1, \dots, a_k \in A\} \ \ \text{and} \ \ A_1 \cdots A_k = \{ a_1 \dots a_k : a_i \in A_i \ (1 \leq i \leq k) \}. \]
Now suppose $|A_i| = N$ for all $1 \leq i \leq k$, for some $N \in \mathbb{N}$. Recent work of the author \cite{Mu2024b} on the sum-product phenomenon combined with the breakthrough work of Gowers--Green--Manners--Tao \cite{GGMT2025} on the polynomial Freiman--Ruzsa conjecture implies that there exists some absolute constant $c'>0$ such that the estimate
\begin{equation} \label{brs}
\max\{|kA_i|, |A_i^{(k)}|\} \gg_k N^{c'( \log k)^{1/8} } 
\end{equation}
holds for all $1 \leq i \leq k$.  Let $S \subseteq \{1,2,\dots, k\}$ be the set of indices $i$ such that $|kA_i|$ achieves the maximum in \eqref{brs} and let $T = \{1,2,\dots, k\} \setminus S$. Note that either $|S| \geq k/2$ or $|T| \geq k/2$. In the former case, we may apply Theorem \ref{prapp} and \eqref{brs} to get that
\[ |A_1 + \dots + A_k| \geq (|kA_1| \dots |kA_k|)^{1/2k} \geq \prod_{i \in S} |kA_i|^{1/2k} \gg_k N^{c'(\log k)^{1/8}|S|/(2k)} \geq N^{c'(\log k)^{1/8}/4}.  \]
One may similarly prove that if $|T| \geq k/2$, then $|A_1 \cdots A_k| \gg_k N^{c'( \log k)^{1/8}/4}$. We record this as the following corollary.

\begin{Corollary} \label{brs2}
    There exists an absolute constant $c''>0$ such that for any $k, N \in \mathbb{N}$ and any finite, non-empty sets $A_1, \dots, A_k \subseteq \mathbb{R}$ of cardinality $N$, one has
\begin{equation} \label{without}
\max\{|A_1 + \dots + A_k|, |A_1  \cdots A_k|\} \gg_k    N^{ c''(\log k)^{1/8}}. 
\end{equation}
    \end{Corollary}

It is worth noting that \cite[Theorem 1.4]{Mu20204c} can be put together with inequality \eqref{brs} to get so-called \emph{low-energy decompositions} with unbounded power savings over $\mathbb{R}$, which can then be combined with Theorem \ref{main} and \eqref{ohtosee} to obtain a quantitatively weaker version of inequality \eqref{without}; see \cite{Mu2024a} for variations of this over $\mathbb{Q}$. For quantitatively stronger sum-product estimates of the form \eqref{brs} over $\mathbb{Q}$, see \cite{PZ2020}, for significant generalisations of such lower bounds to the setting of $1$-dimensional algebraic groups over $\mathbb{C}$, see \cite{HMS2026}, and finally, see \cite{BSSZ2026} for recent groundbreaking upper bound constructions towards the sum-product problem over $\mathbb{R}$.



\subsection*{Outline}
We use \S\ref{fouriersection} to describe the Fourier analytic proof of Theorem \ref{main}. In \S\ref{rep1},
we will provide the combinatorial proof of Theorem \ref{main} along with the proof of Proposition \ref{q1} in \S\ref{rep1}. We prove Theorem \ref{nonab} in \S\ref{rep2} and Theorem \ref{prapp} in \S\ref{prapp2}.

\subsection*{Notation}
We employ Vinogradov notation, that is, we write $Y \ll_{z} X$, or equivalently $Y =O_z(X)$, to mean that $|Y| \leq C_z X$, where $C_z>0$ is some constant depending on the parameter $z$. For every natural number $k \geq 2$ and for every non-empty, finite set $Z$, we use $|Z|$ to denote the cardinality of $Z$. We write $Z^k = \{ (z_1, \dots, z_k)   :   z_1, \dots, z_k \in Z\}$ and we use boldface to denote vectors $\vec{z} = (z_1, z_2, \dots, z_k) \in Z^k$. 

\subsection*{Acknowledgements} 
The author would like to thank Anurag Sahay and Ben Green for helpful comments. The author is supported by a Leverhulme Early Career Fellowship \texttt{ECF-2025-148}. Copilot was used in the process of proving Proposition \ref{q1}; all other ideas were generated by the author. This article is written entirely by the author.


\section{A Fourier analytic proof of Theorem \ref{main}} \label{fouriersection}

    For every $f: G \to \mathbb{C}$, let ${\rm supp}(f)$ denote the support of $f$. Since the set $X = {\rm supp}(\nu) \cup (\cup_{1 \leq i \leq 2d} {\rm supp}(w_i))$ is finite, we can reduce to the case when $G$ is the group generated by $X$. Now, since $G$ is finitely generated, it is isomorphic to some group
    \[ H = \mathbb{Z}^N \times (\mathbb{Z}/N_1 \mathbb{Z}) \times \dots \times (\mathbb{Z}/N_r\mathbb{Z}),\]
    where $N, r, N_1, \dots, N_r \in \mathbb{Z}$ satisfy $N, r\geq 0$ and $N_1, \dots, N_r \geq 1$. Note that additive relations are preserved under isomorphisms, and so, we may assume that $G=H$.  At this point, we can do Fourier analysis, and so, we consider the following isomorphic copy of the dual group
    \[ H^* = \mathbb{T}^N \times (\mathbb{Z}/N_1 \mathbb{Z}) \times \dots \times (\mathbb{Z}/N_r\mathbb{Z}). \]
    Define the measure $\mu$ on $H$ to satisfy
    \[ \int_H f(\vec{x}) d\mu =  \sum_{n_1 \in \mathbb{Z}/N_1\mathbb{Z}} \dots \sum_{n_r \in \mathbb{Z}/N_r\mathbb{Z}} \ \sum_{\vec{k} \in \mathbb{Z}^N} f(\vec{k}, n_1, \dots, n_r) \]
    for all finitely-supported functions $f: H \to \mathbb{C}$. Given $\vec{x} = (\vec{k}, n_1, \dots, n_r) \in H$ and $\vec{\xi} = (\vec{\eta}, m_1, \dots, m_r ) \in H^*$, we define
    \[ \vec{x} \cdot \vec{\xi} = \vec{k} \cdot \vec{\eta} + n_1m_1/N_1 + \dots + n_r m_r/N_r\]
    with $\vec{k} \cdot \vec{\eta}$ denoting the usual dot product in $\mathbb{R}^N$. Moreover, for any finitely-supported function $f: H \to \mathbb{C}$, define its Fourier transform $\hat{f}: H^* \to \mathbb{C}$ as
    \[ \widehat{f}(\vec{\xi}) = \int_H f(\vec{x}) e(\vec{x} \cdot \vec{\xi}) d\mu,  \]
    where we write $e(\theta) = e^{2 \pi i\theta}$ for all $\theta \in \mathbb{R}$. Furthermore, define the measure $\mu^*$ on $H^*$ to satisfy
    \[ \int_{H^*} g(\vec{\xi}) d\mu^* =   (N_1 \dots N_r)^{-1}\sum_{m_1 \in \mathbb{Z}/N_1\mathbb{Z}} \dots \sum_{m_r \in \mathbb{Z}/N_r\mathbb{Z}} \int_{\mathbb{T}^N} g(\vec{\eta}, m_1, \dots, m_r)  d\vec{\eta} \]
    for all bounded functions $g : H^* \to \mathbb{C}$. Orthogonality now implies that
    \begin{align} \label{wcw} 
   E_{2d,\nu}(w_1, \dots, w_{2d})  & =  \int_{H} \dots \int_H w_1(\vec{a}_1) \dots w_{2d}(\vec{a}_{2d}) \nu(\vec{y})\nu(\vec{y}')  \1_{\sum_{i=1}^{d}(\vec{a}_i - \vec{a}_{i+d}) = \vec{y} - \vec{y}'} d\mu(a_1) \dots d\mu(y') \nonumber \\
    & = \int_{H^*} \widehat{w_1}(\vec{\xi})\dots \widehat{w_d}(\vec{\xi}) \overline{\widehat{w_{d+1}}(\vec{\xi}) \dots \widehat{w_{2d}}(\vec{\xi})} |\widehat{\nu}(\vec{\xi})|^2 d\mu^* .
    \end{align}
     H\"{o}lder's inequality gives us the bound
    \begin{equation}  \label{wcw2} 
    E_{2d, \nu}(w_1, \dots, w_{2d}) \leq \prod_{1 \leq i \leq 2d} \bigg(  \int_{H^*} |\widehat{w_i}(\vec{\xi})|^{2d} |\widehat{\nu}(\vec{\xi})|^2 d\mu^* \bigg)^{1/2d}.
    \end{equation}
    Applying orthogonality for each of the individual integrals on the right hand side  delivers the desired bound.


\section{A combinatorial proof of Theorem \ref{main}} \label{rep1}

We begin by presenting our proof of Theorem \ref{main} that utilises Proposition \ref{q1}.

\begin{proof}[Proof of Theorem \ref{main}]
After renormalising both sides of the inequality in \eqref{lmd}, we may assume that $w_i$ satisfies $\sum_{x \in G} w_i(x)^2 =1$ for every $1 \leq i \leq 2d$, and that $\sum_{x \in G} \nu(x)^2 = 1$. For any $x \in G$, let 
\[ r(x) = \sum_{z \in G}\sum_{a_1, \dots, a_d \in G} \nu(z) w_1(a_1) \dots w_d(a_d) \1_{x = a_1 + \dots + a_d + z} \]
and
\[ r'(x) =  \sum_{z' \in G}\sum_{a_1', \dots, a_d' \in G} \nu(z') w_{d+1}(a_1') \dots w_{2d} (a_d') \1_{x = a_1' + \dots + a_d'+ z'}\]
 We can apply Cauchy's inequality to deduce that
\begin{align} \label{slp3}
E_{2d, \nu}(w_1, \dots, w_{2d}) & = \sum_{x \in G} r(x) r'(x) \leq (\sum_{x \in G} r(x)^2 )^{1/2}  (\sum_{x \in G} r'(x)^2 )^{1/2} \\
& = E_{2d, \nu}(w_1, \dots, w_d,w_1, \dots, w_d)^{1/2} E_{2d, \nu}(w_{d+1}, \dots, w_{2d}, w_{d+1}, \dots,w_{2d})^{1/2} \nonumber . 
\end{align}
Thus, it suffices to prove that
\begin{equation} \label{symcase}
    E_{2d,\nu}(w_1, \dots, w_d, w_1, \dots, w_d)  \leq \prod_{1\leq i \leq d} E_{2d,\nu}(w_i)^{1/d}. 
\end{equation}
Let $H_d$ be defined as in \eqref{hdbleh} and let $\vec{u} = (u_1, \dots, u_d) \in H_d$. For this choice of $\vec{u}$, let $\Gamma_{\vec{u}}$ denote the weighted number of solutions to the equation
\begin{equation} \label{slp1}
\sum_{i=1}^{d} (x_i - y_i)  + z = z',
\end{equation}
where for any $1 \leq i \leq d$, we count $x_i, y_i$ with weights $w_j(x_i) w_j(y_i)$ with $j$ being the smallest number such that $i \leq u_1 + \dots + u_j$, and where we count $z,z' \in G$ with weights $\nu(z) \nu(z')$. In particular, for each $1 \leq j \leq d$, there are precisely $2u_j$ variables listed in \eqref{slp1} associated with the weight function $w_j$. For example, observe that
\begin{equation} \label{circuit1}
E_{2d,\nu}(w_1, \dots, w_d, w_1, \dots, w_d) = \Gamma_{(1,\dots,1)} \ \ \text{and} \ \ E_{2d,\nu}(w_i) = \Gamma_{d\vec{e}_i}. 
\end{equation}
Moreover, note that $\Gamma_{\vec{u}} \geq 1$ for every $\vec{u} \in H_d$, since we can always let $x_1, \dots, x_d, z$ have any admissible value and let $y_i = x_i$ for all $1 \leq i \leq d$ and $z' = z$, in which case, we get that
\[ \Gamma_{\vec{u}} \geq (\sum_{z \in G} \nu(z)^2 ) \prod_{1 \leq i \leq d} (\sum_{x \in G} w_i(x)^2 )^{u_i} = 1.  \]

We now claim that for any $\vec{u}, \vec{v} \in H_d$ satisfying $(\vec{u} + \vec{v})/2 \in H_d$, we have
\begin{equation} \label{cauchyschwarz1}
\Gamma_{(\vec{u} + \vec{v})/2 }  \leq \Gamma_{\vec{u}}^{1/2} \Gamma_{\vec{v}}^{1/2} 
\end{equation}
Indeed, let $\vec{z} = (\vec{u} + \vec{v})/2 \in H_d$. Let $1\leq i \leq d$ satisfy $z_i \neq 0$. Let $1 \leq T \leq d$ be such that 
\begin{equation} \label{sp4}
x_{T}, x_{T+1}, \dots, x_{T + z_i-1},y_{T}, y_{T+1}, \dots, y_{T + z_i - 1} 
\end{equation}
are precisely the $2z_i$ variables appearing in equation \eqref{slp1}  which are counted by $\Gamma_{\vec{z}}$ with the weights $w_i(x_T) \dots w_i(y_{T + z_i -1})$. Let the first $u_i$ variables listed in \eqref{sp4} remain on the left hand side in \eqref{slp1} and move the remaining $2z_i - u_i = v_i$ variables in \eqref{slp1} to the right hand side in  \eqref{slp1}. On the other hand, if $1 \leq i \leq d$ is an index such that $z_i =0$, then we must have that $u_i = v_i = 0$ since $u_i, v_i \geq 0$ and $u_i + v_i = 2z_i$. In this case there is nothing to be done. After doing the above procedure for each $1 \leq i \leq d$ and  performing a suitable relabelling, we see that $\Gamma_{\vec{z}}$ counts the number of solutions to the equation
\[ t+ (-1)^{h_1}a_1 + \dots + (-1)^{h_d} a_d = t' + (-1)^{h_1'} b_1 + \dots  + (-1)^{h_d'} b_d \]
where for every $1 \leq i \leq d$ we count $a_i, b_i \in G$ with weights $w_j(a_i)$ and $w_{j'}(b_i)$, with $j$ and $j'$ being the smallest integers such that 
\begin{equation} \label{slp4}
i \leq u_1 + \dots + u_j \ \ \text{and} \ \ i \leq v_1 + \dots + v_{j'}, \end{equation} and where $(h_1, \dots, h_d), (h_1', \dots, h_d')\in \{0,1\}^d$ are some fixed vectors depending on $\vec{u}, \vec{v}$, and where $t,t' \in G$ are counted with weights $\nu(t) \nu(t')$. Applying Cauchy's inequality as in \eqref{slp3}, we see that
\[ \Gamma_{\vec{z}} \leq \Gamma'^{1/2} \Gamma''^{1/2}, \]
where $\Gamma'$ counts solutions to the equation
\[ t_1 +  (-1)^{h_1} a_1 + \dots  + (-1)^{h_d} a_d  =  (-1)^{h_1} a_1' + \dots  + (-1)^{h_d} a_d' + t_2\]
with $a_i,a_i' \in G$ weighted by $w_j(a_i) w_j(a_i')$, where $j$ is the smallest integer satisfying the first condition in \eqref{slp4}, and with $t_1,t_2$ being counted with weights $\nu(t_1) \nu(t_2)$. One can define $\Gamma''$ similarly. We first consider $\Gamma'$, wherein, upon rearranging and relabelling the variables suitably depending on whether $h_i = 0$ or $h_i =1$ for every $1 \leq i \leq d$, we see that $\Gamma'$ is precisely $\Gamma_{\vec{u}}$. One can similarly analyse $\Gamma''$ and prove that $\Gamma'' = \Gamma_{\vec{v}}$. This finishes the proof of our claim.

Now, define the function $f: H_d \to \mathbb{R}$ by setting
\[
f(\vec{u}) = \log \Gamma_{\vec{u}} \]
for every $\vec{u} \in H_d$. We see that $f(\vec{u}) \geq 0$ for all $\vec{u} \in H_d$ since $\Gamma_{\vec{u}} \geq 1$ for all $\vec{u} \in H_d$. Moreover, our claim recorded in \eqref{cauchyschwarz1} is equivalent to the fact that the function $f$ satisfies the condition described in \eqref{disc}. Applying Proposition \ref{q1} gives us the bound
\[ \log E_{2d,\nu}(w_1, \dots, w_{d}, w_1, \dots, w_{d}) =  f(1,\dots, 1) \leq \sum_{1 \leq i \leq d} \frac{f(d \vec{e}_i)}{d} = \sum_{1 \leq i \leq d} \frac{\log E_{2d,\nu}(w_i)}{d} , \]
which is precisely the estimate claimed in \eqref{symcase}.
\end{proof}

We now record the proof of Proposition \ref{q1}.

\begin{proof}
Proposition \ref{q1} is trivially true when $d =1$, and so, we may assume that $d \geq 2$. For any $(u_1, \dots, u_d) \in H_d$, let 
\[ g(u_1, \dots, u_d) = f(u_1, \dots, u_d) - \frac{1}{d} \sum_{i = 1}^d u_i f(d\vec{e}_i). \]
Note that $g(d \vec{e}_i) = 0$ for all $1 \leq i \leq d$. Moreover, for any $\vec{u}, \vec{v} \in H_d$ such that $(\vec{u} + \vec{v})/2 \in H_d$, our hypothesis recorded in \eqref{disc} implies that
\begin{equation} \label{obs}
g((\vec{u} + \vec{v})/2) \leq (g(\vec{u}) + g(\vec{v}))/2.
\end{equation}
We see that proving  \eqref{lvd} is equivalent to proving that $g(\vec{u}) \leq 0$ for all $\vec{u} \in H_d$. We prove the latter via contradiction, and so, suppose there exists some $\vec{u}' \in H_d$  such that $g(\vec{u}') >0$. Since $H_d$ is finite, there exists some $\vec{z} \in H_d$ such that $g(\vec{z}) \geq g(\vec{u})$ for all $\vec{u} \in H_d$. This implies that $g(\vec{z}) \geq g(\vec{u}') > 0$. For any $\vec{v} \in H_d$, let ${\rm supp}(\vec{v})$ denote the set of indices $1 \leq j \leq d$ such that $v_j >0$. Now, fix some $i \in {\rm supp}(\vec{z})$. Our main claim here is that there exists some $\vec{z}' \in H_d$ such that 
\begin{equation} \label{mainclaim}
    g(\vec{z}') = g(\vec{z}) \ \ \text{and} \ \ {\rm supp}(\vec{z}') \subsetneq {\rm supp}(\vec{z}) \ \ \text{and} \ \ z_i' = 0
\end{equation}
This can be iteratively applied to deduce that there exists some $1 \leq j \leq d$ such that $g(d \vec{e}_j) \geq g(\vec{u})$ for all $\vec{u} \in H_d$. This means that $0 = g(d\vec{e}_j) \geq g(\vec{u}') > 0$, dispensing the desired contradiction. Thus, we will now proceed to prove \eqref{mainclaim}.

In order to prove this, observe that since $g(d\vec{e}_j) = 0$ for all $1 \leq j \leq d$, we must have ${\rm supp}(\vec{z}) \neq \{i\}$. Thus, there exists $j \in {\rm supp}(\vec{z}) \setminus \{i\}$ such that the vectors $\vec{v} = \vec{z} - \vec{e}_i + \vec{e}_j$ and $\vec{v}'= \vec{z} + \vec{e}_i - \vec{e}_j$ satisfy $\vec{v}, \vec{v}' \in H_d$. Note that $(\vec{v} + \vec{v}')/2 = \vec{z}$, and so, we may apply \eqref{obs} to deduce that
\[ g(\vec{z}) \leq (g(\vec{v}) + g(\vec{v}'))/2. \]
On the other hand, since $\vec{v}, \vec{v}' \in H_d$, we get that $g(\vec{z}) \geq g(\vec{v}), g(\vec{v}')$, whence the above inequality implies that $g(\vec{z}) = g(\vec{v}) = g(\vec{v}')$. Note that $v_i < z_i$ and that $g(\vec{v}) = g(\vec{z})$. If $v_i = 0$, then we may set $\vec{z}' = \vec{v}$ and we would be done because $v_j > z_j >0$ and $v_k = z_k$ for all $k \in \{1,2,\dots, d\} \setminus \{i,j\}$. Otherwise, we repeat the above argument with $\vec{z}$ replaced by $\vec{v}$. Since $z_i \leq d$, we can repeat this argument at most $d$ times before we obtain some $\vec{z}' \in H_d$ satisfying \eqref{mainclaim}. This finishes the proof of Proposition \ref{q1}.
\end{proof}


\section{Proof of Theorem \ref{nonab}}   \label{rep2}

Suppose $G$ is a finite group and
\begin{equation} \label{blackkeys1}
a_1 a_2^{-1}a_3 a_4^{-1} \dots a_{2d-1} a_{2d}^{-1} = 1
\end{equation}
holds for some $a_1, \dots, a_{2d} \in G$. Let $x_0 \in G$, and define  $x_{2i-1} = x_{2i-2} a_{2i-1}$ and $x_{2i} = x_{2i-1} a_{2i}^{-1}$ for every $1 \leq i \leq d$. Then we have that $x_{2d} = x_0 (a_1 a_2^{-1} \dots a_{2d-1} a_{2d}^{-1}) = x_0$. In particular, considering a  bipartite graph on $G \times G$ where $(x,y) \in G \times G$ lies in the edge set $E$ if $y = xa$ for some $a \in G$, we see that 
\begin{equation} \label{blackkeys2}
(x_0,x_1), (x_2, x_1), (x_2, x_3), \dots, (x_{2d-2}, x_{2d-1}), (x_{0}, x_{2d-1}) \in E . 
\end{equation}
Thus, this set of edges forms a cycle of length $2d$ in the above Cayley graph. In fact, there is a bijection between all $(2d+1)$-tuples $(x_0, a_1, \dots, a_{2d}) \in G^{2d+1}$ such that $a_1, \dots, a_{2d}$ satisfy \eqref{blackkeys1}, with the set of all cycles of length $2d$ described in \eqref{blackkeys2}. Indeed, for any such cycle, we can define $a_{2i-1} = x_{2i-2}^{-1} x_{2i-1}$ and $a_{2i} = x_{2i}^{-1} x_{2i-1}$ for all $1 \leq i \leq d$ with $x_{2d}=x_0$. This means that
\begin{align} \label{astro}
|G|S_{2d}&(w_1, \dots, w_{2d})  = \sum_{x_0 \in G} \sum_{a_1, \dots, a_{2d} \in G} w_1(a_1) \dots w_{2d}(a_{2d}) \1_{a_1 a_2^{-1}a_3 a_4^{-1} \dots a_{2d-1}a_{2d}^{-1} = 1} \nonumber \\
& = \sum_{x_0,x_1, \dots, x_{2d-1} \in G} w_1(x_0^{-1} x_1) w_2(x_2^{-1}x_1) \dots w_{2d-1}(x_{2d-2}^{-1} x_{2d-1}) w_{2d}(x_0^{-1} x_{2d-1}).
\end{align}
Writing the right hand side in \eqref{astro} to be $\Lambda(w_1, \dots, w_{2d})$, our proof of Theorem \ref{nonab} reduces to proving the following lemma.

\begin{lemma} \label{vsch}
    Let $d \in \mathbb{N}$, let $G$ be a finite group, and let $w_1, \dots, w_{2d} : G \to [0, \infty)$ be functions. Then 
    \[ \Lambda(w_1, \dots, w_{2d}) \leq \prod_{1 \leq i \leq 2d} \Lambda(w_i, \dots, w_i)^{1/2d}. \]
\end{lemma}

We note that Lemma \ref{vsch} is a special case of \cite[Corollary 2.6]{Ha2010}, but we present its proof here since it is quite short and since it highlights the similarities between the spectral approach and the Fourier analytic argument from \S\ref{fouriersection}. In this endeavour, we record some notation and preliminaries. Thus given some $n \times n$ matrix $A$ with real entries, we define $A^T$ to be its transpose, and for every $1\leq j,k \leq n$, we denote $A_{j,k}$ to be the entry in the $j^{\rm th}$ row and $k^{\rm th}$ column of $A$. We will also write ${\rm Tr}(A)$ to denote the trace of $A$. Given real number $p \in [1, \infty)$ we denote $\| A\|_{p}$ to be the $p^{\rm th}$ Schatten norm of $A$, that is, 
\[ \| A\|_{p} = (\sum_{1 \leq i \leq n} \sigma_i(A)^p )^{1/p} \]
where $\sigma_1(A), \dots, \sigma_n(A) \geq 0$ are the singular values of $A$. When $d \in \mathbb{N}$, we have
\[ \| A\|_{2d} = ({\rm Tr}((A A^T)^d  ) )^{1/2d}. \]
 Given $n \times n$ matrices $A_1, \dots, A_{2d}$ with real entries, we can apply von Neumann's trace inequality and H\"{o}lder's inequality for Schatten norms to get that
\begin{equation} \label{vnmshcatten}
|{\rm Tr}(A_1 A_2^T \dots A_{2d-1} A_{2d}^T)| \leq \| A_1 A_2^T \dots A_{2d-1} A_{2d}^T \|_1 \leq \prod_{1\leq i \leq 2d} \|A_i\|_{2d}.
\end{equation} 
We are now ready to prove Lemma \ref{vsch}.

\begin{proof}[Proof of Lemma \ref{vsch}]
Let $G = \{x_1, \dots, x_n\}$ for some $n \in \mathbb{N}$. For every $1 \leq i \leq 2d$, define the $n \times n$ matrix $A_i$ by letting $(A_i)_{j,k} = w_i(x_j^{-1} x_k)$ for every $1 \leq j,k \leq n$. 
 Let $A = A_1 A_2^T\dots A_{2d-1} A_{2d}^T$. Note that $A$ is an $n \times n$ matrix and for any $1 \leq j, k \leq n$, one has
\[ A_{j,k} = \sum_{1 \leq i_1, \dots, i_{2d-1} \leq n}   w_1(x_j^{-1} x_{i_1}) w_2( x_{i_2}^{-1} x_{i_1})\dots w_{2d-1}(x_{i_{2d-2}}^{-1} x_{i_{2d-1}}) w_{2d}(x_k^{-1}x_{i_{2d-1}}) .\]
Combining this with \eqref{astro} gives us
\begin{equation} \label{nbd1}
    {\rm Tr}(A_1 A_2^T \dots A_{2d-1} A_{2d}^T) = {\rm Tr}(A) = \sum_{j=1}^n A_{j,j} = \Lambda(w_1, \dots, w_{2d}). 
\end{equation}
Finally, applying \eqref{vnmshcatten}, we obtain the desired conclusion 
\begin{align*}
    |\Lambda(w_1, \dots, w_{2d})| & = |{\rm Tr}(A_1 A_2^T \dots A_{2d-1}A_{2d}^T)| \leq \prod_{1 \leq i \leq 2d} \|A_i\|_{2d} \\
    & = \prod_{1 \leq i \leq 2d} ({\rm Tr}((A_i A_i^T)^d))^{1/2d} = \prod_{1 \leq i \leq 2d} \Lambda(w_i, \dots, w_i)^{1/2d}.  \qedhere
\end{align*}
\end{proof}


\section{Proof of Theorem \ref{prapp}} \label{prapp2}

We will need the following consequence of the Pl\"{u}nnecke--Ruzsa inequality, see \cite[Corollary 6.28]{TV2006}. 

\begin{lemma} \label{pr2}
    Let $K \geq 1$ and let $A,B$ be finite, non-empty subsets of some abelian group such that $|A+B| \leq K|A|$. Then for any $n \in \mathbb{N}$ one has
    \[ |nB| \leq K^n |A|. \]
\end{lemma}

We are now ready to prove Theorem \ref{prapp}. For every $2 \leq j \leq d$, we define $K_j \geq 1$ as
\[ |A_1 +  \dots  + A_j| = K_j |A_1 + \dots + A_{j-1}| .\]
Lemma \ref{pr2} now implies that
\[ |dA_j| \leq K_j^d |A_1 + \dots + A_{j-1}|. \]
Thus, we have that
\[ |A_1 + \dots + A_j| \geq |dA_j|^{1/d} |A_1 + \dots + A_{j-1}|^{1 - 1/d}. \]
We now iteratively apply this to get that
\begin{align*}
|A_1 + \dots + A_d| 
& \geq |dA_d|^{1/d} |A_1 + \dots + A_{d-1}|^{1 - 1/d}  \\
& \geq |dA_d|^{1/d} |dA_{d-1}|^{(1 - 1/d)/d} |A_1 + \dots + A_{d-2}|^{ ( 1- 1/d)^2} \\
& \geq |dA_d|^{1/d} |dA_{d-1}|^{(1-1/d)/d} |d A_{d-2}|^{(1-1/d)^2 /d } \dots |dA_2|^{(1-1/d)^{d-2}/d} |A_1|^{(1-1/d)^{d-1}}
\end{align*}
Now, for all $1 \leq j \leq d$, define $\sigma_j : \{1,2,\dots, d\} \to \{1,2,\dots, d\}$ to be the permutation satisfying $\sigma_j(x) \equiv x+j \ ({\rm mod} \ d)$ for all $1 \leq x \leq d$. Since we are in an abelian group, we see that $A_1 + \dots + A_d = A_{\sigma_j(1)} + \dots + A_{\sigma_j(d)}$ for all $1 \leq j \leq d$. Thus, we get that
\begin{align*}
    |A_1 + \dots + A_d| & = \prod_{1 \leq j \leq d} |A_{\sigma_j(1)} + \dots + A_{\sigma_j(d)}|^{1/d} \\
    & \geq \prod_{1\leq i \leq d} |dA_i|^{\frac{1}{d^2}(1 + (1-1/d) + \dots + (1-1/d)^{d-2})} \\
    & = \prod_{1 \leq i \leq d} |dA_i|^{\frac{1-(1-1/d)^{d-1}}{d}}. 
\end{align*}
Since $\{(1-1/d)^{d-1}\}_{d=2}^{\infty}$ is a decreasing sequence, we have that  $(1-1/d)^{d-1} \leq 1/2$ for all $d \geq 2$. Combining this with the previous estimate finishes the proof of Theorem \ref{prapp}.


\end{document}